\documentclass{amsart}
\usepackage{amsfonts,amssymb,amsmath}
\usepackage{enumerate}
\usepackage{graphicx}
\usepackage{xypic}

\numberwithin{equation}{section}

\def\ca{{\mathcal A}}

\def\cc{{\mathcal C}}

\def\cg{{\mathcal G}}
\def\ch{{\mathcal H}}

\def\cf{{\mathcal F}}
\def\cl{{\mathcal L}}
\def\cm{{\mathcal M}}

\def\cs{{\mathcal S}}
\def\ct{{\mathcal T}}
\def\cu{{\mathcal U}}
\def\cz{{\mathcal Z}}
\def\cx{{\mathcal X}}
\def\cy{{\mathcal Y}}

\def\c{\gamma}
\def\b{\beta}

\def\ve{\varepsilon}

\def\kp{\kappa}

\def\l{\label}
\def\|{|-|}

\def\bz{\mathbb Z}

\def\beg{\begin{gather*}}
\def\eng{\end{gather*}}
\def\begn{\begin{gather}}
\def\eegn{\end{gather}}

\def\beal{\begin{align}}
\def\eeal{\end{align}}

\def\bean{\begin{eqnarray}}
\def\eean{\end{eqnarray}}
\def\bea{\begin{eqnarray*}}
\def\eea{\end{eqnarray*}}

\def\beq{\begin{equation}}
\def\eeq{\end{equation}}

\def\nm{\nonumber}

\def\eea{\end{eqnarray*}}

\def\os{\overset}
\def\er{\eqref}

\def\ot{\otimes}
\def\op{\oplus}

\def\otm{\otimes M}

\def\ofm{\otimes_{G{\mathcal F}} M}

\def\gf{G\cf}
\def\ogf{\otimes_{\gf}}
\def\hogf{\widetilde\otimes_{\gf}}

\def\H{{\rm Hom}}
\def\M{{\rm Map}}
\def\Hg{{\rm Map}_{G\cu}}

\def\mt{\mapsto}

\def\a{\alpha}

\def\2h1{(\alpha_2^H)^{-1}}

\def\cop{\coprod}

\def\chom{\ch {om}}

\def\t{\widetilde}
\def\x0{\{x_0\}}
\def\wh{\widetilde}

\def\on{\operatorname}
\def\dps{\displaystyle}

\def\wt{\widetilde}

\def\ol{\overline}

\def\blt{\bullet}

\newcommand{\calf}{\mathcal{F}}
\newcommand{\calg}{\mathcal{G}}
\newcommand{\call}{\mathcal{L}}
\newcommand{\calr}{\mathcal{R}}

\newcommand{\coker}{\operatorname{coker}}
\newcommand{\hot}{\widetilde{\otimes}}
\newcommand{\lad}{\operatorname{lad}}

\newcommand{\WH}{W\!H}

\newcommand{\id}{\operatorname{id}}

\theoremstyle{plain}
\newtheorem{theorem}[equation]{Theorem}
\newtheorem{proposition}[equation]{Proposition}
\newtheorem{lemma}[equation]{Lemma}

\theoremstyle{definition}
\newtheorem{definition}[equation]{Definition}
\newtheorem{example}[equation]{Example}

\theoremstyle{remark}
\newtheorem{remark}[equation]{Remark}

\begin{document}

\title{Stable equivariant abelianization, its properties, and applications}

\author{Pedro F. dos Santos}

\address{Department of Mathematics, Instituto Superior T\'ecnico, Avenida Rovisco Pais, Lisboa 1049-001, Portugal}
\email{pedfs@math.ist.utl.pt}

\author{Zhaohu Nie}

\address{Department of Mathematics, Penn State Altoona, 3000 Ivyside Park, Altoona, PA 16601, U.S.A.}
\email{znie@psu.edu}


\begin{abstract}
Let $G$ be a finite group. For a based $G$-space $X$ and a Mackey functor $M$, a topological Mackey functor $X\wh\otimes M$ is constructed, which will be called the stable equivariant abelianization of $X$ with coefficients in $M$. When $X$ is a based $G$-CW complex, $X\wh\otimes M$ is shown to be an infinite loop space in the sense of $\cg$-spaces. This gives a version of the $RO(G)$-graded equivariant Dold-Thom theorem. Applying a variant of Elmendorf's construction, we get a model for the Eilenberg-Mac Lane spectrum $HM$. 
The proof uses a structure theorem for Mackey functors and our previous results. 
\end{abstract}

\maketitle

\tableofcontents

\section{Introduction}
In nonequivariant algebraic topology,  a fundamental tool to study singular homology is the abelianization 
functor
$
X\mapsto X\otimes_\calf\bz,
$
from topological spaces to topological abelian groups ($X\otimes_\calf\bz$ is our notation for the free abelian on $X$ appropriately topologized;
see Example~\ref{G-modules}). 
The classical Dold-Thom theorem \cite{dt} asserts 
a natural isomorphism 
$$
\pi_i(X\otimes_\calf\bz)\cong H_i(X;\bz),
$$ 
when $X$ is a CW complex. 
Applying (the reduced version of) this functor to spheres produces a geometric model for the Eilenberg-Mac~Lane 
spectrum $H\bz$. 

In this paper, we construct a functor that will play the role of the abelianization
functor in the context of stable equivariant algebraic topology over a finite group $G$
(which will be fixed throughout the paper). 

First we need to establish some notation. We denote by $\cf,\cu,\ct,$ and $\ca b$  the categories of finite sets, unbased 
(compactly generated) topological spaces, based (compactly generated) topological spaces, and abelian groups. 
The corresponding categories of $G$-objects (given by a monoid homomorphism from $G$ to the Hom monoid of the object)
and equivariant morphisms are denoted by $G\cf, G\cu,G\ct$ (where the base point is $G$-fixed), and $G\on{-Mod}$ ($G$-modules). 
We will denote a Hom space in a topological 
category $\cc$ (such as $\cu$,  $G\cu$, $\ct$, $G\ct$) by $\M_\cc$.

One important distinction between the equivariant and nonequivariant settings is that in the former, 
coefficients are more complicated objects, and differ for homology and cohomology. By definition, a \emph{covariant coefficient system} 
$k$ is a covariant functor  $k\colon G\cf\to \ca b$ that transforms disjoint unions into direct sums, and this is the coefficient system needed for equivariant homology. The corresponding \emph{contravariant coefficient systems} are needed for equivariant cohomology. 

Another related notion  is that of  \emph{based $\cg$-space}, i.e. a contravariant functor $\cx\colon G\cf^{op}\to \ct$ transforming disjoint unions into products. 
The functor category of based $\cg$-spaces is denoted by $\cg\ct$. Naturally,  there is also the notion 
of (unbased) \emph{$\cg$-space}
whose category we denote by $\cg\cu$. 

\begin{remark}
\label{GF&CG}
Let $\cg$ be the orbit category of $G$ - the objects are orbits $G/H$ and the morphisms are $G$-maps. There is a natural inclusion functor $\cg\to G\cf$, since $G$ is finite. Furthermore, each finite $G$-set $S$ can be uniquely written as a disjoint union of $G$-orbits (hence $G/H$ for some $H$ after we choose one point in the orbit). Therefore a $\cg$-space $\cx$ is completely determined by its restriction 
$$
\cx:\cg^{op}\to \ct,
$$ 
since $\cx$ is required to transform disjoint unions to products. We will interchangeably use the two equivalent definitions throughout the paper.
\end{remark}

The importance of the category  $\cg\ct$ stems
from the fact that it is related to $G\ct$ by the \emph{functor of fixed points}: 
\begin{equation}\l{definephi}
\Phi\colon G\ct\to \cg\ct;\ X\mapsto \bigl((S\in \gf)\mapsto \M_{G\cu}(S,X)\cong \M_{G\ct}(S_+,X)\bigr),
\end{equation}
where 
the equivariant mapping spaces  are based by the 
constant map to the base point of $X$. 
(When $S=G/H$, $\Phi X(G/H)=\M_{G\cu}(G/H,X)=X^H$ is the fixed point space.) 
If no risk of confusion arises, we will often abuse  notation and write $X$ for $\Phi X\in \cg\ct$. 
Similarly, there is an unbased version of the  functor of fixed points also denoted $\Phi\colon G\cu\to \cg\cu$. 
We will also use a variant of the coalescence functor $\Psi\colon\cg\cu\to G\cu$ introduced by Elmendorf in \cite{elmendorf}. 
Up to homotopy it is the right adjoint of $\Phi$, and allows us to work in the  category $\cg\cu$,
with its standard model structure \cite[Chapter VI]{may}, and then translate results back
to $G\cu$.

The natural replacement for singular homology in the equivariant setting is  \emph{Bredon homology} \cite{bredon}.
In  \cite{n}  the second 
author constructed   an equivariant abelianization functor (not stable yet). 
For each  covariant coefficient system $k$ and each $G$-space $X$,  he defined 
a topological abelian group $X\otimes_{\gf} k$
(the original notation was $\cg X\ogf k$), by the following \emph{coend} construction
\begin{equation}\l{coend}
X\ogf k=\coprod_{S\in \gf} \M_{G\cu}(S,X)\times k(S)/\approx,
\end{equation}
where the equivalence relation is generated by 
\begin{equation}\l{equivrel}
\M_{G\cu}(S,X)\times k(S)\ni (\phi f^*,\kp)\approx (\phi, f_* \kp)\in \M_{G\cu}(T,X)\times k(T)
\end{equation}
for a map $f:S\to T$ in $G\cf$ with $\phi f^*=\phi\circ f$ and $f_*=k(f)$. 
In \cite[Theorem 1]{n}, he showed  that for a   $G$-CW complex $X$, 
there is a natural isomorphism
\begin{equation}\l{nieiso}
\pi_i(X\otimes_{\gf} k)\cong H_i^G(X;k),
\end{equation}
where the right hand side denotes the Bredon equivariant homology of $X$ with coefficients in $k$. 

Thus, for Bredon homology with a general covariant coefficient system, the question of finding an
equivariant abelianization functor
is satisfactorily settled. However, for the important class of covariant coefficient systems which also have suitably coupled contravariant functoriality to form some objects called  
\emph{Mackey functors} (we will recall the definition in Section 2), Bredon homology has a lot more structure since:
\begin{enumerate}[(i)]
\item  it can be enhanced to a theory with values in the category of Mackey functors~\cite{dieck};
\item it can be enhanced to an $RO(G)$-graded theory~\cite{may}.
\end{enumerate}

From the viewpoint of stable equivariant homotopy theory, Mackey functors are the coefficient systems needed for an ordinary stable equivariant cohomology theory~\cite{may}. 
Our goal in this paper is precisely to construct a version of the abelianization functor adapted to 
the stable equivariant homotopy theory, which we call the stable equivariant abelianization functor. 

To this end, we start with  a Mackey functor $M$ and a $G$-space $X$ and enhance the topological abelian group  
\er{coend} to a \emph{topological Mackey functor}  $X\ot M$ (see Definition~\ref{definition}). We achieve this 
by following closely the procedure used in \cite{dieck} to endow  $H_*(X;M)$ with a canonical Mackey  functor structure.

We then proceed to verify expected properties, like the fact that for a point $x_0$ one gets
$x_0\otimes M=M$, and the existence of a reduced version of the construction  $X\hot M$ 
(Definition~\ref{hot}) satisfying
$$
X\ot M=M\op X\hot M,
$$
for $X$ based.
In Example~\ref{G-modules} we show that our construction generalizes that given in \cite{ds} where 
the special case of a Mackey functor associated to a  $G$-module was treated.

Being a topological Mackey functor, $X\wh\otimes M$ is in particular a based $\cg$-space (we consider a topological abelian 
group to be based at $0$). Thinking of $\cg$-spaces as generalizations of $G$-spaces,  
$X\widetilde\otimes M$ is a reasonable candidate for the  stable equivariant abelianization functor.  
In Section~3 we  show that this is indeed the right candidate: for a based $G$-CW complex $X$, 
\begin{enumerate}[(i)]
\item  we introduce the notion of $\varOmega$-$\cg$-spectrum  (Definition \ref{cgstuff}) and prove in Theorem~\ref{loopspace} that the correspondence 
$V\mapsto (\varSigma^VX)\wt\otimes M$ defines an $\varOmega$-$\cg$-spectrum, denoted $(\varSigma^\infty X)\wt\otimes M$ (here $V$ ranges over 
finite dimensional $G$-representations, $S^V:=V\cup \infty$ is the $V$-sphere, and $\varSigma^V X=S^V\wedge X$ is the corresponding suspension);
\item we establish the following $RO(G)$-graded version of the Dold-Thom theorem (Theorem~\ref{rogdt}) 
$$
\pi_V^\cg(X\wh\otimes M):=[\Phi S^V,X\hot M]_{\cg\ct}\cong \t H_V^G(X;M),
$$
where  the $[-,-]_{\cg\ct}$ denotes based homotopy classes of based $\cg$-maps
and the right hand side denotes the  $RO(G)$-graded equivariant homology of $X$. 
\end{enumerate}

Our strategy to  prove these results is to use a structure theorem for Mackey functors of 
Greenlees and May~\cite{gm} to reduce to the case of a Mackey functor obtained from a $G$-module 
(see Example~\ref{G-modules}), which was treated in~\cite{ds}.  

In Section~4 we show that applying a variant of Elmendorf's  coalescence functor $\Psi$ \cite{elmendorf} to the  
$\varOmega$-$\cg$-spectrum $(\varSigma^\infty X)\wt\otimes M$ gives an $\Omega$-$G$-spectrum. 
In particular, taking $X=S^0$ we get a new model for the equivariant Eilenberg- Mac~Lane spectrum
$H\!M$. This model differs from the previously known models \cite{cw} in that it does not require any
stabilization with respect to the representation spheres.

\medskip
\noindent{\bf Acknowledgement.} We would like to thank Paulo Lima-Filho  and Gustavo Granja for many helpful discussions related to this work. 

\section{Construction and examples}

In this section, we associate a topological Mackey functor $X\otimes M$ to  each pair $(X,M),$ where $X$ is a $G$-space and $M$ is Mackey functor $M$. 
We also introduce a reduced version of this construction for a based $G$-space $X$, denoted $X\hot M$. We study the properties of the bifunctors
$(X,M)\mapsto X\otimes M$ and $(X,M)\mapsto X\hot M$, and we discuss some examples.

Let us  start by recalling the definition of a Mackey functor.
\begin{definition}\label{def-Mackey}
A \emph{Mackey functor} $M$ consists of a pair $(M^*,M_*)$ of functors $M^*: \gf^{op}\to \ca b$ and $M_*:\gf\to \ca b$ 
with the same values on objects, which we denote by $M$, such that 
\begin{enumerate}
\item $M$ transforms disjoint unions into direct sums;
\item For each pullback diagram 
$$\xymatrix{
A\ar[r]^f\ar[d]_g & B\ar[d]^h\\
C\ar[r]^k & D}$$
in $\gf$, there is a commutative diagram in $\ca b$
\begin{equation}\l{commm}
\xymatrix{
M(A)\ar[r]^{M_*(f)} & M(B)\\
M(C)\ar[r]^{M_*(k)}\ar[u]^{M^*(g)} & M(D)\ar[u]_{M^*(h)}.
}
\end{equation}
\end{enumerate}

We denote the category of Mackey functors by $\cm k$. This is an abelian category, with kernels and cokernels defined using the abelian structure of $\ca b$. 

A topological Mackey functor is defined similarly with $\ca b$ replaced by the category $\ct\ca b$ of topological abelian groups. We denote the category of topological Mackey functors by $\ct\cm k$. 
\end{definition}

Now we  introduce our candidate for the abelianization functor 
for unbased $G$-spaces.
\begin{definition}
\l{definition} 
For a $G$-space $X$ and a Mackey functor $M=(M_*,M^*)$, we define a \emph{topological Mackey functor} $X\otm$ as follows. 

On the object level, for a finite $G$-set $S\in G\cf$, we define 
\begin{equation}\l{construct}
(X\otimes M)(S):=(X\times S)\ot_{G\cf} M_*=:(X\times S)\ot_{G\cf} M
\end{equation}
as in~\er{coend}.


For a morphism $f:S\to T$ in $G\cf$, we write $f_*=(X\otm)_*(f)$ and $f^*=(X\otm)^*(f)$ for simplicity and define them as follows. 
$$f_*:(X\otimes M)(S)=(X\times S)\ot_{G\cf} M\os{(\id\times f)\ogf \id}\longrightarrow (X\times T)\ot_{G\cf} M=(X\otimes M)(T)$$
is defined by the covariant functoriality of the coend construction $-\ot_{G\cf} M$~\cite[(2.4)]{n}. 

We now define 
\bea
f^*: (X\otm)(T)=(X\times T)\ofm\to (X\times S)\ofm=(X\otm)(S).
\eea
By~\er{coend}, an element on the left is represented by $(\gamma, c)\in \M_{G\cu}(C,X\times T)\times M(C)$ for some $C\in \gf$. 
We form the following pullback diagram in the category $G\cu$
\begin{equation}\l{pullback}
\xymatrix{
B\ar[d]^F\ar[r]^{\beta} & X\times S\ar[d]^{\id\times f}\\
C\ar[r]^\c & X\times T.
}
\end{equation}
Note that $B\in G\cf$. 
Define 
\begin{equation}\l{f^*}
f^*:(X\otm)(T)\to (X\otm)(S);\ [(\c, c)]\mapsto [(\beta, M^*(F)(c))],
\end{equation}
where $(\beta,M^*(F)(c))\in \Hg(B,X\times S)\times M(B)$. (Here and after, we denote the equivalence class of an element by $[-]$.)

\end{definition}

In Lemma~\ref{check} below we show that $f^*$ is well defined, and that $X\otm$ is a topological Mackey functor. 

\begin{remark}\label{Pedro-general} In Definition~\ref{definition} it would perhaps be more precise to denote the resulting 
topological Mackey functor by $\Phi X\otimes M$,   because it is
the $\cg$-space  $\Phi X$ that is used to define the values of $X\otimes M$. Indeed,  we have 
$(X\otimes M)(S)=(\Phi X\times \Phi S)\otimes_{G\cf}M$. In a similar fashion we could  define 
$\cx\otimes M$, for any $\cg$-space $\cx$. However, in what follows, we will only be interested in applying
this construction to the case where $\cx=\Phi X$, for some $G$-space $X$, and therefore we chose to simplify the notation by 
dropping the $\Phi$.
\end{remark}



\begin{lemma}\l{check} With Definition~\ref{definition}, $X\ot M$ becomes a topological Mackey functor. 
\end{lemma}

\begin{proof} We first prove that $f^*$ in~\er{f^*} is well defined, i.e., the definition is independent of the choice of the representative $(\c,c)\in \Hg(C,X\times T)\times M(C)$. Choose another representative $(\c',c')\in \Hg(C',X\times T)\times M(C')$. Without loss of generality, by~\er{equivrel}, we assume that there is a $G$-map $h:C\to C'$, such that the following diagram
\begin{equation}\l{base}
\xymatrix{
C\ar[rr]^\c \ar[rd]^h & & X\times T\\
 & C'\ar[ru]^{\c'} & 
}
\end{equation}
commutes, i.e., $\c=\c'\circ h=\c'h^*$, and 
\begin{equation}\l{a'}  
c'=M_*(h)(c).
\end{equation}

Consider the following diagram
$$\xymatrix{
B\ar[rr]^\beta\ar[dd]^F\ar[rd]^{\exists !\ H} & & X\times S\ar[dd]^{\id\times f}\\
 & B'\ar[ru]^{\beta'}\ar[dd]^(.3){F'} & \\
C\ar'[r]^{\c}[rr]\ar[rd]^h & & X\times T,\\
 & C'\ar[ru]^{\c'} & 
}$$
where the bottom is diagram~\er{base}, the back and the right front faces are pullback diagrams of the form~\er{pullback}. Then by the universality of $B'$, there exists a unique $H:B\to B'$ to make the left front face and the top commute. One can easily see that the left front face is also Cartesian since both the fibers of $F$ and $F'$ are the same as the fibers of $\id\times f$. 

Then by definition~\er{f^*}, we need to show 
$$(\beta, M^*(F)(c))\in \Hg(B,X\times S)\times M(B)$$
and 
$$(\beta', M^*(F')(c'))\in \Hg(B',X\times S)\times M(B')$$
are equivalent. Since the top is commutative, $\beta=\b'\circ H=\b' H^*$. In view of~\er{equivrel}, we only need to show that $M_*(H)(M^*(F)(c))$ and $M^*(F')(c')=M^*(F')(M_*(h)(c))$ by~\er{a'} are equal. 
This follows from the left front face being Cartesian and the property~\er{commm} of the Mackey functor $M$. 

It can be checked that $f^*$ is a continuous homomorphism, and that $X\otm$ satisfies the conditions for a Mackey functor in Definition~\ref{def-Mackey}.
\end{proof}

\begin{proposition}[Functoriality]\l{funct} The construction introduced in Definition~\ref{definition} gives a bifunctor
$$\ot:G\cu\times \cm k\to \ct\cm k;\ (X,M)\mapsto X\otimes M,$$
such that $X\ot -$ is also exact: For an exact sequence of Mackey functors
\begin{equation}\l{mackey-exact}
0\to M\to N\to P\to 0,
\end{equation}
there is an exact sequence of topological Mackey functors
\begin{equation}\l{topexact}
0\to X\ot M\to X\ot N\to X\ot P\to 0.
\end{equation}

\end{proposition}

\begin{proof} The functoriality in both $G\cu$ and $\cm k$ is easy to see from the definition. 


To prove the exactness of \er{topexact}, we need to show that for any $S\in G\cf$, the following sequence of topological abelian groups 
\begin{equation}\l{groupexact}
0\to (X\times S)\ogf M\to (X\times S)\ogf N\to (X\times S)\ogf P\to 0
\end{equation}
is exact. 
For an element in $(X\times S)\ogf M$, one can choose a representative $(\b,b)\in \Hg(B,X\times S)\times M(B)$ with $\beta$ injective. (Otherwise, one can replace $\beta$ with the inclusion $i:\on{Im}(\b)\hookrightarrow X\times S$ and apply the equivalence relation~\er{equivrel}.) For an injective $\beta$ with $B\neq \varnothing$, 
\begin{equation}\l{observation}
[(\b,b)]=0\in (X\times S)\ot_{G\cf} M\Leftrightarrow b=0\in M(B).
\end{equation}
The exactness of \er{groupexact} then easily follows from this observation and \er{mackey-exact}. 
\end{proof}



\begin{example}\l{apoint} Applying our construction to the space $x_0$ of one point, one canonically recovers $M$, i.e., 
$$x_0\otm= M.$$
This fact is rather straightforward, and we leave the details to the reader. 
\end{example}

We will now give a reduced version of the functor $X\otimes M$, defined for a based space $(X,x_0)$.
Recall that the base point $x_0$ is $G$-fixed. Consider the natural maps 
$$
x_0\os i\to X\os p\to x_0.
$$
Example~\ref{apoint} shows that $x_0\ot M=M$. Since  $p\circ i=\id$, it follows 
by functoriality (Proposition~\ref{funct}) that one has natural maps 
$$
M=x_0\ot M\os {i_*}\to X\ot M\os {p_*}\to x_0\ot M=M,
$$
such that 
\begin{equation}\l{splitting}
p_*\circ i_*=\id. 
\end{equation}

\begin{definition}\l{hot} For a based $G$-space $X$ with base point $x_0$ and a Mackey functor $M$, define the reduced topological Mackey functor 
$$X\hot M=
\on{coker}(i_*:x_0\ot M\to X\ot M).$$
The functor 
$$\hot:G\ct\times \cm k\to \ct\cm k;\ (X,M)\mapsto X\hot M$$
is our \emph{stable equivariant abelianization functor}.
\end{definition}

By the splitting \er{splitting}, one naturally has a direct sum decomposition of Mackey functors
\begin{equation}\l{directsum}
X\ot M=M\op X\hot M.
\end{equation}



\begin{proposition}\l{cofiber}
A cofibration sequence of $G$-spaces 
$$Y\os i \hookrightarrow X\os q \to X/Y$$
gives rises to a short exact sequence
of topological Mackey functors
\begin{equation}
\label{cofibration-seq}
0\to Y\otimes M \os{i_*}\to X\otimes M \os{q_*}\to (X/Y)\hot M\to 0.
\end{equation}

\end{proposition}

\begin{proof} In view of \er{observation}, we see that $i_*$ is injective since $i$ is a closed inclusion, $q_*$ is surjective since $q$ is, and the sequence is exact in the middle since the inverse image of the base point under $q:X\to X/Y$ is $Y$. 
\end{proof}

One can write out the above definition of $X\hot M$ in terms of its components as follows. 
For a (general) based space $X$, first define the reduced topological abelian group
\begin{equation*}
X\hot_{G\calf}M:=\coker(i_*\colon x_0\ot_{G\calf}M\to X\ot_{G\calf}M).
\end{equation*}
One then has for a finite $G$-set $S$,
\begin{equation}\l{onsets}
(X\hot M) (S) \cong (X\wedge S_+)\hot_{G\calf}M.
\end{equation}


We now give two important examples of $X\otm$ and $X\hot M$ in the case where $M$ belongs to 
two special classes of  Mackey functors. These examples will be used extensively in  Section~3.


\begin{example}
\label{G-modules}
A $G$-module $A$ determines a Mackey functor $\calr A$ whose value on a 
finite $G$-set $S$ is  the abelian group
$
{\calr}A(S)=\H_G(S, A)
$,
where the abelian group structure comes from the target $A$.   
For a $G$-map $f\colon S\to T$, the restriction and transfer maps are given by
\begin{align*}
\calr {A}^*(f)&\colon{\calr} A(T)\to {\calr} A(S);\ \phi\mapsto \phi\circ f, \\
{\calr} A_*(f)&\colon{\calr} A(S)\to {\calr} A(T);\ \psi\mapsto \left(t\mapsto \sum_
{s\in f^{-1}(t)}\psi(s)\right).
\end{align*}

Given a $G$-space $X$, it is natural to consider the  $G$-module $\oplus_{x\in X}A$ generated by $X$ with coefficients in $A$. It has a natural structure of topological $G$-module, which can be seen most clearly from its description as the 
coend $X\otimes_\calf {A}$.  
Here $\calf$ denotes the category
of finite sets,  $X$ denotes the contravariant functor $\calf^{op}\to G\mathcal{U}\colon
S\mapsto \M_\cu(S,X)$,  and $A$ denotes the covariant functor 
$\calf\to G\mathcal{U}\colon S\mapsto \H(S,A)$ with  
$$A_*(f\colon S\to T)\colon \H(S,A)\to \H(T,A);\ \psi\mapsto \left(t\mapsto \sum_{s\in f^{-1}(t)}\psi(s)\right).$$
See~\cite{n} for more details. 

For a based $G$-space $(X,x_0)$,  there is a reduced version of this construction, given by 
$$
X\widetilde{\otimes}_\calf A:=
\operatorname{coker}(A=x_0\otimes_\calf A \to X\otimes_\calf  A).
$$ 
Then similar to \er{directsum}, one has the following direct sum decomposition 
\begin{equation}\l{directsum2}
X\ot_\cf A=A\op X\hot_\cf A.
\end{equation}

The functor $X\mapsto X\hot_\calf A$ has a natural  structure of \emph{functor with smash products} (FSP)
given by
\begin{equation}\l{defineL}
L_{Y,X}\colon Y\wedge (X\hot_\calf A)\to (Y\wedge X)\hot_{\calf} A;\ 
(y,c)\mapsto (f_y\hot_\calf \id_A)(c)
\end{equation}
where $f_y\colon X\to Y\wedge X$ denotes the function $x\mapsto y\wedge x$. 
This FSP is studied in~\cite{ds} (where it is denoted $X\mapsto A\hot X$). In particular, 
it is shown there that the corresponding $G$-prespectrum  $\{S^V\hot_\calf A\}_V$ is 
an Eilenberg-Mac Lane $\Omega$-spectrum for the Mackey functor $\calr{A}$.

Applying the functor $\calr$ to the topological $G$-module $X\otimes_\calf A$, we obtain
a topological Mackey functor $\calr(X\otimes_\calf A)$. 
(If we only consider the contravariant functoriality, then $\calr$ is the same as the $\Phi$ in~\er{definephi}.) 
$\calr(X\otimes_\calf A)$ is closely related to the topological Mackey functor 
$X\otimes\calr A$ defined in Definition~\ref{definition}. 
Indeed, for a finite $G$-set $S$, the forgetful functor  $ G\calf\to \calf$  
induces a natural map
$$
(X\otimes \calr A)(S)=(X\times S)\otimes_{G\calf}\calr{A}\to (X\times S)\otimes_\calf A,
$$
and it is shown in~\cite[Proposition 2]{n} that this map is an isomorphism onto the fixed point set
\begin{equation}\l{oldone}
(X\times S)\ogf {\mathcal R} A\os\sim\to ((X\times S)\otimes_\calf A)^G.
\end{equation}

There is a natural $G$-homeomorphism (see~\cite{ds} for more details)
$$
\M_\cu(S,X\otimes_\calf A)\os\sim\to (X\times S)\otimes_\cf A;\  f\mapsto \sum_{s\in S}L_{S_+,X_+}(s,f(s)), 
$$
where the left is the mapping space with the $G$-action of conjugation whose fixed point set consists of equivariant maps, and the $L$ on the right is as in~\er{defineL}. 
Therefore on the level of fixed point sets, we have a homeomorphism 
\begin{equation}\l{mapandtensor}
((X\times S)\otimes_\cf A)^G\simeq  \M_{G\cu}(S,X\otimes_\calf A)={\mathcal R}(X\otimes_\cf A)(S).
\end{equation}

It can be checked that the above defines 
a natural isomorphism of topological Mackey functors 
\begin{equation}
\label{dfn:rho}
 \varrho\colon X\otimes \calr{A}\to \calr(X\otimes_\calf A).
\end{equation}
\end{example}


The class of Mackey functors considered in the previous example
is very restrictive but  it admits an  important generalization introduced in~\cite{gm}, which 
we discuss now.

\begin{example}
\l{realexample} 
Given  a subgroup $H<G$, let $\WH=NH/H$ be its Weyl group.  There is a functor
$$
{\mathcal R} \colon \WH\on{-Mod}\to \cm k,
$$
such that,  for a $\WH$-module $A$ and a finite $G$-set $S$, one has
\begin{equation}\l{easier}
{\mathcal R} A(S)=\H_{\WH}(S^H,A).
\end{equation}
For a $G$-map $f:S\to T$, we consider the restriction $f^H: S^H\to T^H$ which is a $\WH$-map. Then 
\begin{align*}
{\mathcal R} A^*(f)&\colon {\mathcal R} A(T)\to {\mathcal R} A(S);\ \phi\mapsto \phi\circ f^H\\
{\mathcal R} A_*(f)&\colon {\mathcal R} A(S)\to {\mathcal R} A(T);\ \psi\mapsto \left(t\mapsto \sum_
{s\in (f^H)^{-1}(t)}\psi(s)\right)
\end{align*}

Whenever it is necessary to make the pair $(G,H)$ explicit
we write $\calr_H^G$ instead of $\calr$.
In particular, if $H=\{1\}$ and  $A$ is a $G$-module then $\calr_H^GA$ coincides with   
the Mackey functor $\calr A$ defined in Example~\ref{G-modules}.

The importance of the images of the functors $\calr_H^G$ (for all $H<G$) as a  class of Mackey functors  comes from a result of~\cite{gm}, which states that this class generates all Mackey functors in a precise
sense that we recall in Section~3. 





The functor ${\mathcal R} $ has a left adjoint~\cite{gm}
$$
\cl: \cm k\to \WH\on{-Mod};\ M\mapsto M(G/H),
$$ 
where one notes that $M(G/H)$ is a $\WH$-module since $\H_{\gf}(G/H,G/H)=\WH$ acts on it.

This adjunction has  an obvious topological analogue:
$$\xymatrix{
{\rm topological}\ \WH{\rm -modules} \ar@<.5ex>[r]^(.475){\mathcal R} & {\rm topological\ Mackey\ functors}\ar@<.5ex>[l]^(.525)\cl.
}$$

Now, associated to a $G$-space $X$ and a $\WH$-module $A$, we have two topological 
Mackey functors: $X\otimes\calr{A}$ and $\mathcal{R}(X^H\otimes_\calf A)$, where   $X^H\ot_\calf A$
is a topological $\WH$-module defined in Example~\ref{G-modules} (with $G=\WH$).
The next proposition shows that there is a natural isomorphism
between these two functors, generalizing the isomorphism $\varrho$~\er{dfn:rho} in  
Example~\ref{G-modules}.

\begin{proposition}\l{identify} For a $\WH$-module $A$ and a $G$-space $X$, one has a natural isomorphism of topological Mackey functors
\begin{equation*}
\varrho\colon X\ot {\mathcal R} A\xrightarrow{\cong} {\mathcal R} (X^H\ot_\calf A).
\end{equation*}

For a based $X$, taking off from both sides the trivial factor ${\mathcal R} A$ in view of~\er{directsum} and~\er{directsum2}, one also has an isomorphism for the reduced version:
$$\widetilde\varrho\colon X\widetilde\ot {\mathcal R} A\xrightarrow{\cong} {\mathcal R} (X^H\widetilde\ot_\calf A).$$
\end{proposition}


For the proof, we need the following lemma. 

\begin{lemma}\l{RLspace} Let $G$ be a finite group, $H$ a subgroup of $G$, and $\WH$ the Weyl group of $H$ in $G$. Then one has the following two adjoint functors:
$$R: G\cu\to \WH\cu;\ X\mapsto X^H,$$
and 
$$L: \WH\cu\to G\cu;\ Y\mapsto G/H\times_{\WH} Y,$$
such that 
$$\M_{G\cu}(LY,X)=\M_{\WH\cu}(Y,RX).$$
\end{lemma}

\begin{proof} The counit is
\begin{equation}\l{unit}
\varepsilon: LRX=G/H\times_{\WH} X^H\to X;\ (gH,x)\mapsto gx,
\end{equation}
and the unit is 
\begin{equation}\l{counit}
\eta: Y\to RLY=(G/H\times_{\WH} Y)^H;\ y\mapsto (eH,y).
\end{equation}

One can easily see that the following composition 
\begin{gather}
RX\os{\eta R}\to RLRX\os{R \varepsilon}\to RX:\l{identity}\\
X^H\to (G/H\times_{\WH} X^H)^H\to X^H;\nm\\
x\mapsto (eH,x)\mapsto x\nm
\end{gather}
is the identity. Similar composition for $LY$ is also the identity. 
\end{proof}


\begin{proof}[Proof of Proposition~\ref{identify}]
We only prove the unreduced version, and the reduced version clearly follows by~\er{directsum} and~\er{directsum2}. 


First note that, for a finite $G$-set $S$, we have 
\begin{align} 
&{\mathcal R}(X^H\otimes_\cf A)(S)\os {\er{easier}} =\M_{\WH\cu}(S^H,X^H\ot_\cf A)\l{same3}\\
\os{\er{mapandtensor}}=&((X^H\times S^H)\ot_\cf A)^{\WH} \os{\er{oldone}}=(X^H\times S^H)\ot_{\WH\cf} {\mathcal R}^{\WH}_{\{1\}}A\nm\\
\os{\er{coend}}=&\coprod_{U\in \WH\cf} \M_{\WH\cu}(U,X^H\times S^H)\times \H_{\WH}(U,A)/\approx.\nm
\end{align}





For a finite $G$-set $S$, we define a morphism $\varrho\colon X\ot {\mathcal R} A\to \calr (X^H\ot_\calf A)$
as the following composition:
\begin{equation}\l{deff}
\begin{array}{lcl}
 & &(X\ot {\mathcal R} A)(S)=(X\times S)\ogf {\mathcal R} A\\
&=&{\displaystyle \coprod_{T\in G\cf}} \M_{G\cu}(T,X\times S)\times \H_{\WH}(T^H,A)/\approx\\
&\os{R\times \id}\to&{\dps \cop_{RT\in \WH\cf}} \M_{\WH\cu}(RT,RX\times RS)\times \H_{\WH}(T^H,A)/\approx\\
&\os{\er{same3}}\to& (X^H\times S^H)\ot_{\WH\cf} {\mathcal R}^{\WH}_{\{1\}}A=\calr(X^H\ot_\calf A)(S).
\end{array}
\end{equation}


Also define a morphism $\varsigma\colon\calr(X^H\ot_\calf A)\to X\ot\calr A$ by
\begin{equation*}
\begin{array}{lcl}
 & &\calr(X^H\ot_\cf A)(S)\os{\er{same3}}=(X^H\times S^H)\ot_{\WH\cf} {\mathcal R}^{\WH}_{\{1\}}{A}\\
&
=&{\dps \cop_{U\in \WH\cf}} \M_{\WH\cu}(U,R(X\times S))\times \H_{\WH}(U,A)/\approx\\
&\os {L\times \eta_*}\to & {\dps \cop_{LU\in G\cf}} \M_{G\cu}(LU,LR(X\times S))\times \H_{\WH}(RLU,A)/\approx\\
&\os {(\ve\circ)\times \id}  \to & {\dps \cop_{LU\in G\cf}} \M_{G\cu}(LU,X\times S)\times \H_{\WH}(RLU,A)/\approx\\
&\to & (X\times S)\ogf {\mathcal R} A=(X\ot {\mathcal R} A)(S).
\end{array}
\end{equation*}
Here the maps $R$, $L$, $\eta$ and $\ve$ are as in Lemma~\ref{RLspace}, and the map 
\begin{gather*}
\eta_*:\H_{\WH}(U,A)\to \H_{\WH}(RLU,A);\\
a\mapsto \left((u'\in RLU)\mt \sum_{u\in \eta^{-1}(u')\subset U} a(u)\right)
\end{gather*}
is the ``transfer" map for $\eta:U\to RLU$~\er{counit}. 


\begin{lemma} Both $\varrho$ and $\varsigma$ are well-defined, and are natural transformations of topological Mackey functors. 
\end{lemma}
\begin{proof}[Proof of the lemma] After unraveling all the definitions, the well-definedness and the naturality with respect to covariancy of the Mackey functors are easy to check. The naturality with respect to the contravariancy of Mackey functors (see Definition~\ref{definition}) is more involved, and uses the fact that the $R$ in Lemma \ref{RLspace}, being a right adjoint, preserves Cartesian diagrams. We leave the 
details to the interested reader. 
\end{proof}






Let us now  show that both the compositions $\varsigma\circ\varrho$ and $\varrho\circ\varsigma$ 
are the identity. Fix a finite $G$-set $S$. 

First we check the equality $\varsigma\circ\varrho=\id$. 
Pick 
$$
(\a,a)\in \M_{G\cu}(T,X\times S)\times \H_{\WH}(T^H,A).
$$ 
Then 
$$
\varrho(\a,a)=(R\a,a)\in \M_{\WH\cu}(RT,R(X\times S))\times \H_{\WH}(T^H,A),
$$ 
and 
$$(\a',a')=(\varsigma\circ\varrho)(\a,a)=(\ve\circ LR\a,\eta_*a)\in \M_{G\cu}(LRT,X\times S)\times \H_{\WH}(RLRT,A).
$$ 
One has the following commutative diagram 
$$\xymatrix{
LRT\ar[r]^(.4){LR\a}\ar[d]^\ve & LR(X\times S)\ar[d]^\ve\\
T\ar[r]^(.4)\a & X\times S
}$$
by the naturality of $\ve$~\er{unit}. Therefore $\a'=\ve\circ LR\a=\a\circ\ve=\a\ve^*$. By~\er{equivrel}, one has
$$(\a',a')=(\a\ve^*,a')\approx (\a,\ve_*a').$$
To illustrate the situation, we draw the following diagram (which we don't claim to be commutative): 
$$\xymatrix{
RT\ar[rrd]^a\ar[d]^{\eta R} & & \\
RLRT\ar[rr]^{a'=\eta_*a}\ar[d]^{R\ve} & & A.\\
RT\ar[rru]_{\ve_*a'} & &
}$$
Then it is easy to see that 
$$\ve_*a'=\ve_*\eta_*a=a,$$
by~\er{identity}. 

Finally we check the equality $\varrho\circ\varsigma=\id$. Pick 
$$
(\b,b)\in \M_{\WH\cu}(U,RX\times RS)\times \H_{\WH}(U,A).
$$ 
Then 
$$
\varsigma(\b,b)=(\ve\circ L\b,\eta_*b)\in \M_{G}(LU,X\times S)\times \H_{\WH}(RLU,A),
$$ 
and 
\begin{align*}
&(\b',b')=(\varrho\circ\varsigma)(\b,b)=(R(\ve\circ L\b),\eta_*b)\\
\in & \M_{\WH\cu}(RLU,RX\times RS)\times \H_{\WH}(RLU,A).
\end{align*} 
One has the following commutative diagram 
$$\xymatrix{
U\ar[r]^(.4)\b\ar[d]^\eta & RX\times RS\ar[d]^{\eta R}\\
RLU\ar[r]^(.4){RL\b} & RL(RX\times RS).
}$$
by the naturality of $\eta$~\er{counit}. Therefore $\eta R\circ \b=RL\b\circ \eta$. Composing both sides with $R\ve$ from the left, one sees in view of~\er{identity} that 
$$\b=R\ve\circ RL\b\circ \eta=R(\ve\circ L\b)\circ \eta=\b'\circ \eta=\b'\eta^*.$$
Therefore by~\er{equivrel}, one has
$$(\b,b)=(\b'\eta^*,b)\approx (\b',\eta_*b)=(\b',b').$$ 
\end{proof}
\end{example}

The following result will be used throughout, which is a generalization of both~\cite[Proposition 7]{gm} and~\cite[Lemma V.3.1]{may}. 
\begin{proposition}
\l{whcg} 
Let $H<G$ be a subgroup. There are two adjoint functors
\begin{gather*}
{\mathcal L}: \cg\ct\to \WH\ct;\ \cx\mapsto \cx(G/H),\\
{\mathcal R}: \WH\ct\to \cg\ct;\ Y\mapsto \bigl(S\mapsto \M_{\WH\ct}(S^H_+,Y)\bigr)
\end{gather*}
such that 
$$\M_{\WH\ct} ({\mathcal L}\cx,Y)\cong \M_{\cg\ct}(\cx,{\mathcal R}Y).$$
\end{proposition}

\begin{proof} The proof is similar to those for the above-mentioned results.
In particular, the counit $\cl{\mathcal R} Y\to Y$ is the identity, and the map 
\begin{equation}\l{(G/H)}
\M_{\cg\ct}(\cx,{\mathcal R}Y)\to \M_{\WH\ct} ({\mathcal L}\cx,Y)
\end{equation}
is obtained by applying $\cl$, which is taking the $(G/H)$  component. 
\end{proof}

\section{$\varOmega$-$\cg$-spectra and equivariant homology}



It is well-known that many of the concepts of equivariant homotopy theory can be carried over to the context of $\cg$-spaces.
We start this section by introducing the notions of $\cg$-prespectrum and $\varOmega$-$\cg$-spectrum, as generalizations of the more classical notions of $G$-prespectrum and $\Omega$-$G$-spectrum (see the fundamental work~\cite{lms} about these more classical objects). Then we  use the construction of the 
topological Mackey functor  $X\hot M$ (for a based $G$-space $X$ and a Mackey functor $M$) to define a $\cg$-prespectrum $(\varSigma^\infty X)\hot M$.
In Theorem~\ref{loopspace} we show that $(\varSigma^\infty X)\hot M$ is an $\varOmega$-$\cg$-spectrum when $X$ is a based $G$-CW complex. 

We start by defining some basic structures of $\cg\ct$ in order to formulate our results. 

\begin{definition} \emph{Smash products} in the category $\cg\ct$
are defined, in view of Remark~\ref{GF&CG}, by 
$$(\cx\wedge \cy)(G/H):=\cx(G/H)\wedge \cy(G/H).$$
\emph{Internal Hom's} in $\cg\ct$ are defined by 
\begin{equation}\l{internalhom}
\ch om(\cx,\cy)(S):=\M_{\cg\ct}(\cx\wedge \Phi(S_+),\cy)
\end{equation} 
for $S\in G\cf$. (We will often abuse notation by omitting $\Phi$.)
\end{definition}

\begin{proposition} One has the following adjunction
\begin{equation}\l{cgadj}
\M_{\cg\ct} (\cx\wedge \cy,\cz)=\M_{\cg\ct} (\cx,\ch om(\cy,\cz)).
\end{equation}
\end{proposition}

\begin{proof} For an orbit $G/H$, define the unit to be 
\begin{align*}
\eta:& \cx(G/H)=\M_{\cg\ct}(\Phi(G/H_+),\cx) \text{ (Yoneda lemma)}\\
\to &\chom (\cy,\cx\wedge \cy)(G/H)=\M_{\cg\ct}( \Phi(G/H_+)\wedge \cy,\cx\wedge \cy); \\
f\mapsto & f\wedge \id;
\end{align*}
and the counit to be
\begin{align}
\mu: & (\chom(\cy,\cz)\wedge \cy)(G/H)=\M_{\cg\ct}(\cy\wedge \Phi(G/H_+),\cz)\wedge \cy(G/H) \label{cg-counit} \\
\to & \cz(G/H);\ h\wedge y\mapsto  h_{G/H}(y\wedge \id).
\nm
\end{align}

Then one can check the usual conditions for adjunction. 
\end{proof}


\begin{lemma} The functor $\Phi$ \er{definephi} 
preserves smash products 
\begin{equation}
\label{phi-prod}
\Phi(X\wedge Y)=\Phi X\wedge \Phi Y,
\end{equation}
and internal Hom's
\begin{equation}\l{internalhom1}
\Phi  \M_\ct(X,Y)= \chom(\Phi X,\Phi Y),
\end{equation}
where $\M_\ct$ denotes  the internal Hom for $G$-spaces with conjugate $G$-action.
\end{lemma}

\begin{proof} \er{phi-prod} follows from 
$$
(X\wedge Y)^H=X^H\wedge Y^H.
$$

There is a functor \cite[Lemma V.3.1]{may}
\begin{equation}\label{deftheta}
\Theta:\cg\ct\to G\ct;\ \cx\to \cx(G/e),
\end{equation}
which is left adjoint to $\Phi$ in \er{definephi}. Clearly $\Theta\Phi X=X$. Therefore 
\begin{equation}\l{fully-faithful}
\M_{\cg\ct}(\Phi X,\Phi Y)\cong \M_{G\ct}(\Theta\Phi X,Y)=\M_{G\ct}(X,Y).
\end{equation}

On a finite $G$-set $S$, \er{internalhom1} is defined by 
\begin{align*}
&\Phi  \M_\ct(X,Y)(S)\os{\er{definephi}}=\M_{G\ct}(S_+,\M_\ct(X,Y))\\
=&\M_{G\ct}(X\wedge S_+,Y)\os{\er{fully-faithful}}=\M_{\cg\ct}(\Phi(X\wedge S_+),\Phi Y)\\
\os{\er{phi-prod}}=&\M_{\cg\ct}(\Phi X\wedge \Phi(S_+),\Phi Y)\os{\er{internalhom}}=\chom(\Phi X,\Phi Y)(S),
\end{align*}
where the second equality is by an obvious adjunction in $G\ct$. 
\end{proof}

Moreover,
$\cg\ct$  has a model category structure in which the weak equivalences and fibrations are defined component-wise \cite[Chapter VI]{may}. Therefore $\Phi$ preserves the model structures by definition. 



\begin{definition}\l{cgstuff}
Fix a complete $G$-universe ${\bf U}$ and a cofinal set $\ca$ of indexing spaces, which are defined to be finite dimensional sub $G$-spaces of ${\bf U}$ (see  \cite[Chapter XII]{may}). 
Define a $\cg$-\emph{prespectrum} to be a collection of $\cg$-spaces $\{\cx_V\}_{V\in \ca}$ such that for $W\in \ca$, one has the following structure map 
\begin{equation}\l{structuremap}
\sigma_{V,W}:\Phi S^W\wedge \cx_V\rightarrow \cx_{V+W}
\end{equation}
as a map of $\cg$-spaces. The structure maps $\sigma$ are required to satisfy $\sigma_{V,0}=\id$ and the expected transitivity condition: 
$$
\xymatrix{\Phi S^U\wedge \Phi S^W\wedge\cx_V \ar[r]^{\id\wedge\sigma_{V,W}}\ar[d]^{{\er{phi-prod}}} & \Phi S^U\wedge \cx_{V+W}\ar[d]^{\sigma_{V+W,U}}\\
\Phi S^{U+W}\wedge\cx_V\ar[r]^{\sigma_{V,U+W}} &\cx_{V+W+U}.
}
$$

For a $\cg$-space $\cx$, we call $\varOmega^W\cx:=\chom(\Phi S^W,\cx)$ the $W$th \emph{loop space} of $\cx$ as a $\cg$-space.
We call the $\cg$-prespectrum $\{\cx_V\}_{V\in \ca}$ an $\varOmega$-$\cg$-\emph{spectrum} if the adjoint map
$$\tau_{V,W}: \cx_V\to \varOmega^W \cx_{V+W}$$
associated to~\er{structuremap} by \er{cgadj} is a $\cg$-weak equivalence.
\end{definition}

Let $X$ be a based $G$-space, and $M$ a Mackey functor. We consider the following natural $\cg$-prespectrum $(\varSigma^\infty X)\widetilde\otimes M$ with its $V$th space as 
the topological Mackey functor $(\varSigma^V X)\widetilde\otimes M$, which is thus a $\cg$-space. Understanding $X$ generally (i.e., sometimes as $\varSigma^V X$), there is the natural structure map
$$\psi:\Phi S^W\wedge (X\widetilde\otimes M)\to (\varSigma^W X)\widetilde\otimes M,$$
which is defined as follows. For an orbit $G/H$, in view of \er{onsets}, we define
\begin{gather}
\psi_{G/H}:\M_{G\ct} ({G/H}_+,S^W)\wedge ((X\wedge {G/H}_+)\hogf M)\to (\varSigma^{W}X\wedge {G/H}_+)\hogf M;\nm\\
\alpha\wedge \xi\mapsto (((\alpha\circ pr_2)\wedge \id)\hogf \id)(\xi)
\label{defnpsi}
\end{gather}
by the functoriality of our coend construction, where 
$$(\alpha\circ pr_2)\wedge \id:X\wedge {G/H}_+\to \varSigma^{W} X\wedge {G/H}_+$$
is the $G$-map defined by $x\wedge t\mapsto\alpha(t)\wedge x\wedge t$.

Since we are interested in the adjoint of $\psi$ under the adjunction \er{cgadj}:
\begin{equation}
\l{loopmap}
\phi: X\widetilde\otimes M\to \varOmega^W((\varSigma^W X)\widetilde\otimes M),
\end{equation}
we now write it out in detail. Given a finite $G$-set $T$, the value of $\phi$ at $T$ is a map
$$\phi_T:(X\widetilde\otimes M)(T)\to \varOmega^W((\varSigma^W X)\widetilde\otimes M)(T),$$
which is natural in $T$. By the Yoneda lemma and the definition of internal Hom \er{internalhom}, 
$\phi_T$ can be described as a map
$$\phi_T: \M_{\cg\ct}(\Phi(T_+),X\widetilde\otimes M)\to \M_{\cg\ct}(\Phi S^W\wedge \Phi(T_+), (\varSigma^W X)\widetilde\otimes M),$$
such that the image of $\beta\in \M_{\cg\ct}(T_+,X\widetilde\otimes M)$ is the following composition:
$$\Phi S^W\wedge \Phi(T_+)\os{\id\wedge \beta}\longrightarrow \Phi S^W\wedge (X\widetilde\otimes M)\os\psi\to (\varSigma^W X)\widetilde\otimes M.$$

The following result describes one of the most important properties of our stable equivariant abelianization functor.
\begin{theorem}
\l{loopspace} 
For a based $G$-CW complex $X$, the $\cg$-prespectrum $(\varSigma^\infty X)\hot M$ is an $\varOmega$-$\cg$-spectrum, i.e., the adjoint map \er{loopmap} is a $\cg$-weak equivalence.
\end{theorem}

\begin{proof} The strategy of the proof is to use a structure theorem for Mackey functors by Greenlees and May \cite{gm} to reduce to the case 
that $M={\mathcal R} A$, where $A$ is a $W\!H$-module, as in Example~\ref{realexample}. Using Propositions \ref{identify} and \ref{whcg}, 
we will further reduce this 
to the case handled in \cite{ds} and recalled in Example~\ref{G-modules} with the group replaced by $\WH$. The case of a 
general  Mackey functor $M$ is then  completed by an application of the 5-lemma. 

Let $M={\mathcal R} A$ be as in Example~\ref{realexample}. Since $\cg$-weak equivalences are defined component-wise, we need to show that, for each finite $G$-set $T$, 
the map $\phi_T$ in the top row of  the following commutative diagram is a weak equivalence of spaces (for notational simplicity, we will omit the $\Phi$'s in this proof): 
\begin{equation}
\label{thm-loopspace-diagram}
\xymatrix{
\M_{\cg\ct}(T_+,X\widetilde\otimes {\mathcal R} A)\ar[r]^(.4){\phi_T}\ar[d]^{\er{identify}} & \M_{\cg\ct}(S^W\wedge T_+,(\varSigma^WX)\hot\calr A)\ar[d]^{\er{identify}}\\
\M_{\cg\ct}(T_+,{\mathcal R}(X^H\widetilde\otimes_\cf A))\ar[r]^(.4){\tau_T}\ar[d]^{\er{whcg}} & \M_{\cg\ct}(S^W\wedge T_+, 
\calr((\varSigma^{W^H}X^H)\hot_\cf A))\ar[d]^{\er{whcg}}\\
\M_{\WH\ct}(T^H_+,X^H\widetilde\otimes_\cf A)\ar[r]^(.4){\chi_T} & \M_{\WH\ct}(S^{W^H}\wedge T_+^H,\varSigma^{W^H}X^H\widetilde\otimes_\cf A).
}
\end{equation}
The vertical arrows in the above diagram are homeomorphisms for a general space $X$ given by Propositions \ref{identify} and \ref{whcg}, and
the maps $\tau_T$ and $\chi_T$ are completely determined by the commutativity of the diagram. 

For each $\gamma\in\M_{\cg\ct}(T_+,{\mathcal R}(X^H\hot_\cf A))$, the commutativity of diagram \eqref{thm-loopspace-diagram} implies that the value $\tau_T(\gamma)$ is a composite of the form
\begin{equation}\label{def-tau-gamma}
S^W\wedge T_+\os{\id\wedge \gamma}\to S^W\wedge\calr(X^H\hot_\cf A)\os\upsilon\to \calr((\varSigma^{W^H}X^H)\hot_\cf A),
\end{equation}
where  $\upsilon$ is determined by the commutativity of the following diagram:  
$$
\xymatrix{
S^W\wedge (X\widetilde\otimes {\mathcal R} A)\ar[r]^\psi\ar[d]^{\id\wedge \wt\varrho} & (\varSigma^WX)\hot \calr A\ar[d]^{\wt\varrho}\\
S^W\wedge {\mathcal R}(X^H\widetilde\otimes_\cf A)\ar[r]^\upsilon & {\mathcal R}((\varSigma^{W^H}X^H)\hot_\cf A).
}
$$
Recall that $\wt\varrho$ is the isomorphism defined in Proposition~\ref{identify}.  On an orbit $G/K$, this reads
$$
\xymatrix{
\M_{G\ct}(G/K_+,S^W)\wedge ((X\wedge G/K_+)\hot_{G\cf} \calr A)\ar[r]^(0.57)\psi\ar[d]^{\id\wedge\wt\varrho_{G/K}} & (\varSigma^WX\wedge {G/K}_+)\hot_{G\cf} \calr A\ar[d]^{\wt\varrho_{G/K}}\\
\M_{G\ct}({G/K}_+,S^W)\wedge \M_{\WH\ct}({(G/K)}_+^H, X^H\hot_{\cf} A)\ar[r]^(0.57)\upsilon & \M_{\WH\ct}({(G/K)}_+^H, (\varSigma^{W^H}X^H)\hot_{\cf} A).
}
$$
Unraveling the definitions of $\psi$ in \eqref{defnpsi} and  $\varrho$ in \er{deff}, it is easy to see that for 
$\alpha\in\M_{G\ct}({G/K}_+, S^W)$, $\xi\in\M_{\WH\ct}({(G/K)}^H_+,X^H\hot_{\calf}A)$ and $u\in {(G/K)}_+^H$, we have
\begin{equation}
\label{eq-formula-upsilon}
\upsilon(\alpha\wedge\xi)(u)=(f_{\alpha(u)}\hot_\calf\, \id_A)(\xi(u)),
\end{equation}
where $f_{\alpha(u)}\colon X^H\to \varSigma^{W^H}X^H$ is the map $x \mapsto \alpha(u)\wedge x$.


Next we  compute the map $\chi_T$ in diagram \eqref{thm-loopspace-diagram}. Denoting by $\delta=\lad(\gamma)\in \M_{\WH\ct}(T^H_+,X^H\wt\otimes_\cf A)$ the left adjunct
of  $\gamma\in\M_{\cg\ct}(T_+, \calr(X^H\hot_\cf A))$ under the adjunction $\call\vdash\calr$ in Proposition \ref{whcg}, we have
$\chi_T(\delta)=\chi_T(\lad(\gamma)) =\lad(\tau_T(\gamma))$, by the required commutativity of diagram \eqref{thm-loopspace-diagram}. 
Since the counit of the adjunction is the identity \er{(G/H)}, we have
$$
\lad(\tau_T(\gamma))=\call(\tau_T(\gamma))\os{\eqref{def-tau-gamma}}= \call(\upsilon)\circ \call(\id\wedge\gamma).
$$
Therefore, $\chi_T(\delta)$ is given by 
$$
S^{W^H}\wedge T_+^H\os{\id\wedge \delta}\to S^{W^H}\wedge (X^H\hot_\cf A)\os{\call(\upsilon)}\to (S^{W^H}\wedge X^H)\hot_\cf A.
$$
Applying the formula  \eqref{eq-formula-upsilon}, we obtain $\call(\upsilon)(s\wedge\xi)=(f_s\hot_\calf \id_A)(\xi)$, where 
$f_s$ denotes the map $X^H\to S^{W^H}\wedge X^H$ given by $x\mapsto s\wedge x$. 

We conclude that $\chi_T$ is composition with  the adjoint map 
$$
X^H\hot_\calf A\to\Omega^{W^H} ((\varSigma^{W^H}X^H)\hot_\calf A)
$$
of the  $\WH$-prespectrum $
\{V\mapsto \varSigma^VX^H\hot_\calf A\}.$
Since it  is shown in \cite{ds} that 
this is an $\Omega$-$\WH$-spectrum for a $G$-CW complex $X$, it follows  $\chi_T$ is a weak equivalence for each $T$.
Therefore $\phi$ \er{loopmap} is $\calg$-weak equivalence.   

Now for the general case of $M$, we use the structure theorem for Mackey functors of \cite{gm}. This result states that the class of Mackey
functors containing the functors $\{\calr A \}$, for all subgroups $H<G$ and all $WH$-modules $A$, and satisfying the $2$ out of $3$ property for short exact sequences  contains all
Mackey functors. 
Since we have shown that the class of Mackey functors for which \er{loopmap} is a $\cg$-weak equivalence contains Mackey functors of the form $\calr A$, 
it suffices now to show that this class
satisfies the $2$ out of $3$ property for short exact sequences.

Let 
$$
0\to M\to N\to P\to 0
$$
be a short exact sequence of Mackey functors. First we show that, for a $G$-CW complex $X$, one gets a fibration sequence of $\cg$-spaces
$$
X\ot M\to X\ot N\to X\ot P.
$$
This means that, for each finite $G$-set $T$, the  resulting sequence of topological abelian groups (see \er{onsets})
$$
(X\wedge T_+)\hot_{G\calf} M\to (X\wedge T_+)\hot_{G\calf} N\to (X\wedge T_+)\hot_{G\calf} P
$$
is a fibration sequence of topological spaces. To prove this, we use the following detour into simplicial sets and the realization functor. 
Let $X_T$ denote $X\wedge T_+$ and 
$\cs X_T$ denote the total singular simplicial set of $X_T$. Then one has an exact sequence of simplicial abelian groups 
$$
0\to \cs X_T\hot_{G\calf} M\to \cs X_T\hot_{G\calf} N\to \cs X_T\hot_{G\calf} P\to 0.
$$
(This is a simplicial version of Proposition \ref{funct}.) 
By basic simplicial homotopy theory \cite[III.2.10]{gj}, this is then a fibration sequence. Since geometric realization is a Quillen equivalence between the 
categories of simplicial sets and topological spaces \cite{gj}, we have the following diagram 
\begin{equation}
\label{diagram-|-|}
\xymatrix{
|\cs X_T\hot_{G\calf} M|\ar[r]\ar[d] & |\cs X_T\hot_{G\calf} N|\ar[r]\ar[d] & |\cs X_T\hot_{G\calf} P|\ar[d]\\
X_T\hot_{G\calf} M\ar[r] & X_T\hot_{G\calf} N\ar[r] & X_T\hot_{G\calf} P,
}
\end{equation}
where the first row is a fibration sequence. For a $G$-CW complex $X$, the vertical maps are $G$-homotopy equivalences \cite{n}. Therefore the second row is also a fibration sequence. 

Since internal Hom's clearly preserve fibration sequences, we have the following diagram of fibration sequences
$$\xymatrix{
X\ogf M\ar[r]\ar[d]^{\phi} & X\ogf N\ar[r]\ar[d]^{\phi} & X\ogf P\ar[d]^{\phi}\\
\varOmega^W((\varSigma^W X)\widetilde\otimes M)\ar[r] & \varOmega^W((\varSigma^W X)\widetilde\otimes N)\ar[r] & \varOmega^W((\varSigma^W X)\hot P)
}
$$
If two out of the three vertical maps are weak equivalences, it follows by the $5$-lemma that so is the third.
This concludes the proof of the theorem.
\end{proof}



We can now prove  the following $RO(G)$-graded version of the classical Dold-Thom theorem, as an application of our stable equivariant abelianization functor.

\begin{theorem}
\l{rogdt} 
Let $X$ be a based $G$-CW complex. There is a natural isomorphism
$$
\pi_V^\cg(X\wh\otimes M):=[\Phi S^V,X\hot M]_{\cg\ct}\cong \t H_V^G(X;M),
$$
where 
the $[-,-]_{\cg\ct}$ denotes based homotopy classes of $\cg$-maps,  
and the right hand side is the $RO(G)$-graded equivariant homology of $X$ with coefficients in $M$. 
\end{theorem}
\begin{proof}
For a $G$-CW pair $(X,Y)$, consider the functors
$$\wt{\bf h}^G_V(X,Y;M)=\pi_V^\cg((X/Y)\hot M)=[\Phi S^V,(X/Y)\hot M]_{\cg\ct}.$$
Now we verify that the $\wt{\bf h}_*^G$ satisfy the axioms for an $RO(G)$-graded theory with coefficients in $M$. 

The $G$-\emph{homotopy axiom} and the \emph{excision axiom} are easy to check (cf. \cite[Cor. 2.7]{ds}). 
For a cofibration sequence 
$$Y\os i \to X\os p \to X/Y,$$
one has a \emph{fibration sequence} 
of topological Mackey functors 
$$Y\hot M\os {i_*} \to X\hot M\os {p_*} \to (X/Y)\hot M.$$
The proof again uses a detour into the simplicial category through geometric realization and the simplicial version of Proposition \ref{cofiber} (cf. the end of the proof of Theorem \ref{loopspace}). Hence one has the associated long exact sequence in the $\wt{\bf h}_*^G$ theory. 

The \emph{dimension axiom} follows from Example~\ref{apoint}. 

The only new $RO(G)$-\emph{graded suspension axiom} now follows from Theorem~\ref{loopspace}.
\end{proof}


\section{Equivariant Eilenberg-Mac~Lane spectra}

The categories of $G$-spaces and $\cg$-spaces are related by the fixed point functor $\Phi\colon G\cu\to \cg\cu$ (cf. \er{definephi}).
The coalescence functor $\Psi\colon\cg\cu\to G\cu$ defined by Elmendorf \cite{elmendorf} 
shows that, up to weak equivalence, every $\cg$-space is
the fixed point system of a $G$-space. In this Section we define a variant of Elmendorf's functor for both the unbased and the based situations, and study the relation of $\Psi$ with smash products and internal Hom's. Application of $\Psi$ turns an $\varOmega$-$\cg$-spectrum, for example the $(\varSigma^\infty X)\hot M$ defined in Section~3, to an $\Omega$-$G$-spectrum. In the particular case that $X=S^0$, we get a model
for the equivariant Eilenberg-Mac~Lane spectrum $H\!{M}$. 

First we introduce a variant of Elmendorf's construction using the category $G\cf$ of finite $G$-sets instead of the orbit category $\cg$, for the benefit of the existence of finite products (see Proposition \ref{psi&prod} below).

\begin{definition}[{cf. \cite[proof of Theorem~1]{elmendorf}}] 
\label{variant}
Let $J:G\cf\to G\cu$ be the inclusion functor. For $\cx\in \cg\cu$, $\Psi \cx$ is a \emph{$G$-space} defined by 
$$
\Psi \cx:=B(\cx,G\cf,J)=|B_\bullet(\cx,G\cf,J)|
$$ 
where $B_\bullet(-,-,-)$ denotes the triple bar construction:
$B_\bullet(\cx,G\cf,J)$ is a simplicial $G$-space whose space of of $n$-simplices is 
$$
\{(x,\underline f,\alpha)|x\in \cx(S_0);\underline f=S_0\os{f_1}\gets S_1\os{f_2}\gets\cdots \os{f_n}\gets S_n;\alpha\in S_n\},
$$ 
where the $f_i$ are $G$-maps, with the usual face (by composition or functoriality of $\cx$ and $J$) and degeneracy (by insertion of identity) maps. The $G$-action on $B_n(\cx,\cg,J)$ is induced from the one on the images of $J$. 
\end{definition}

The following proposition shows that a $\cg$-space $\cx$ is, up to weak equivalence, the fixed point system of the $G$-space $\Psi\cx$. 

\begin{proposition}[{cf. \cite[Theorem~1]{elmendorf}}]
\label{still-good}
There is a natural transformation $\varepsilon\colon\Phi\Psi\to \id$ such that
for each $\cx\in \cg\cu$ the map $\varepsilon_\cx\colon\Phi\Psi\cx\to \cx$ is a weak equivalence.
In particular $\Psi$ preserves weak equivalences.
\end{proposition}


\begin{proof} For an orbit $G/H$, the natural transformation $\ve$ has value
\begin{align}
\label{vexg/h}
\varepsilon_\cx(G/H)\colon & \Phi\Psi\cx(G/H)=\M_{G\cu}(G/H,B(\cx,G\cf,J))\\
= & B(\cx,G\cf,\H_{G\cf}(G/H,-))\to \cx(G/H),\nonumber
\end{align}
which is defined by pullback for each simplex. Here the second equality follows from the fact that the $G$-action on $B(\cx,G\cf,J)$ is induced from the one on the images of $J$, and that $G/H$ is an orbit. 
It is standard \cite[\S V.2]{may} that $\varepsilon_\cx(G/H)$ is a strong deformation retract.

For the second statement, assume that $f:\cx\to \cy$ is a weak equivalence of $\cg$-spaces. We now prove that $\Psi f:\Psi\cx\to \Psi\cy$ is a weak equivalence of $G$-spaces. By definition, we need to show that $\Phi\Psi f:\Phi\Psi\cx\to \Phi\Psi\cy$ is a weak equivalence of $\cg$-spaces. This follows from the commutativity of the following diagram, 
$$
\xymatrix{
\Phi\Psi\cx\ar[r]^{\Phi\Psi f}\ar[d]^{\ve_\cx} & \Phi\Psi\cy\ar[d]^{\ve_\cy}\\
\cx\ar[r]^f & \cy,
}$$
by the naturality of $\ve$, and the fact that $\ve_\cx$, $\ve_\cy$ and $f$ are weak equivalences. 
\end{proof}

\begin{remark} 
\label{no-big-diff}
It can be checked that our variant in Definition \ref{variant} is the same as Elmendorf's construction in \cite{elmendorf} up to homotopy. 
\end{remark}

Our variant enables us to show the following relation between $\Psi$ and products. 

\begin{proposition} 
\label{psi&prod}
For $\cx,\cy\in \cg\cu$, there is a homotopy equivalence of $G$-spaces 
\begin{equation}
\label{def-varpi}
\varpi\colon\Psi\cx\times\Psi\cy\to\Psi(\cx\times\cy),
\end{equation}
with a homotopy inverse 
$$
\Delta: \Psi (\cx\times \cy)\to \Psi\cx\times\Psi\cy.
$$
defined in \eqref{Delta} below.
\end{proposition}

\begin{proof} Define $\varpi$ as the geometric realization of 
\begin{gather*}
B_\bullet(\cx,G\cf,J)\times B_\bullet(\cy,G\cf,J)\to B_\bullet(\cx\times\cy, G\cf, J):\\
\left((\cx(S_0)\ni x,f_1,\cdots,f_n, \alpha\in S_n),
(\cy(T_0)\ni y, g_1,\cdots,g_n,\beta\in T_n)\right)
\mapsto \\
((\cx\times\cy)(S_0\times T_0)\ni(\cx(p_1)(x),\cy(p_2)(y)),f_1 \times g_1,\cdots,f_n\times g_n, 
(\alpha,\beta)\in S_n \times T_n),
\end{gather*}
where $p_1\colon S_0\times T_0\to S_0$ and $p_2\colon S_0\times T_0\to T_0$ denote the projections.

Define 
$$
\Delta=\Psi(pr_1)\times \Psi(pr_2): \Psi(\cx\times\cy)\to \Psi \cx\times \Psi \cy
$$
by the functoriality of $\Psi$ for the obvious projections. More concretely, $\Delta$ is the geometric realization of 
\begin{gather}
\label{Delta}
B_\bullet(\cx\times\cy, G\cf, J)\to B_\bullet(\cx,G\cf,J)\times B_\bullet(\cy,G\cf,J):\\\notag
((x,y)\in (\cx\times\cy)(S_0),f_1,\cdots,f_n,\alpha)\mapsto\\\notag
((x\in \cx(S_0),f_1,\cdots,f_n,\alpha),(y\in \cy(S_0),f_1,\cdots,f_n,\alpha)).
\end{gather}

It can be checked that $\varpi$ and $\Delta$ are homotopy inverses of each other by constructing homotopies on the simplicial space level using projections and diagonals at appropriate places. 
\end{proof}



We now define a version of the coalescence functor for based $\cg$-spaces, still denoted $\Psi\colon\cg\ct\to G\ct$.
\begin{definition} For $\cx\in \cg\ct$, the \emph{based $G$-space} $\Psi \cx$ is defined to be the geometric realization of
a based $G$-simplicial space, as follows:
\[
\Psi \cx=\overline{B}(\cx,G\cf, J):=|\overline{B}_\bullet(\cx,G\cf, J)|,
\]
where $\overline{B}_\bullet(\cx,G\cf, J)$ is the based simplicial $G$-space defined by
\[
\overline{B}_n(\cx,G\cf, J)=B_n(\cx,G\cf, J)/B_n(*,G\cf, J),
\]
where $*$ denotes the base point of $\cx$. 
\end{definition}





We have the following relations of the functor $\Psi$ with smash products and internal Hom's.

\begin{lemma}
\label{descend}
For $\cx,\cy\in \cg\ct$, there is a based version of the construction in Proposition \ref{psi&prod}
$$
\varpi\colon\Psi\cx\wedge\Psi\cy\to\Psi(\cx\wedge\cy),
$$
which is a weak equivalence of $G$-spaces. 
\end{lemma}

\begin{proof} Consider the natural composition
$$
\Psi\cx \times \Psi\cy \os\varpi\to \Psi(\cx\times \cy)\to \Psi(\cx\wedge \cy),
$$
which is the geometric realization of 
$$
B_\bullet(\cx,G\cf,J)\times B_\bullet(\cy,G\cf,J)\to B_\bullet(\cx\times\cy, G\cf,J)\to B_\bullet(\cx\wedge \cy, G\cf,J).
$$

It is clear that 
$$
B_\bullet(\cx,G\cf,J)\times B_\bullet(*,G\cf,J)\to B_\bullet(\cx\times *,G\cf,J)\to B_\bullet(*,G\cf,J)
$$ 
under the above composition, and similarly for $B_\bullet(*,G\cf,J)\times B_\bullet(\cy,G\cf,J)$. 

Then we define the based $\varpi$ to be the following composition
\begin{align*}
 & \Psi \cx\wedge \Psi \cy= |\ol B_\blt (\cx,G\cf,J)|\wedge |\ol B_\blt(\cy,G\cf,J)|=|\ol B_\blt (\cx,G\cf,J)\wedge \ol B_\blt (\cy,G\cf,J)| \\
=& |(B_\blt (\cx,G\cf,J)\times B_\blt (\cy,G\cf,J))/((B_\blt (\cx,G\cf,J)\times B_\blt (*,G\cf,J))\cup (B_\blt (*,G\cf,J)\times B_\blt (\cx,G\cf,J)))| \\
\to & |B_\blt(\cx\wedge \cy,G\cf,J)/B_\blt(*,G\cf,J)|=|\ol B_\blt(\cx\wedge \cy,G\cf,J)|=\Psi(\cx\wedge \cy).
\end{align*}

By definition, to show that the based $\varpi$ is a weak equivalence, we need to show that 
$$
\Phi\varpi:\Phi\Psi\cx \wedge \Phi\Psi\cy {\os{\er{phi-prod}}=} \Phi(\Psi\cx\wedge \Psi\cy)\to \Phi\Psi(\cx \wedge \cy)
$$
is a weak equivalence of $\cg$-spaces. This follows from the commutativity of the following diagram
\begin{equation}
\label{eq-varomega-varepsilon}
\xymatrix{
\Phi\Psi\cx \wedge \Phi\Psi\cy\ar[rd]^{\varepsilon_\cx\wedge\varepsilon_\cx}\ar[rr]^{\Phi\varpi} & &\Phi\Psi(\cx \wedge \cy)\ar[ld]^{\varepsilon_{\cx\wedge\cy}}\\
&\cx\wedge \cy, & 
}
\end{equation}
and the fact that $\varepsilon_{\cx\wedge\cy}$ and $\ve_\cx\wedge\ve_\cy$ are weak equivalences. Recall that $\ve_\cx(G/H)$ \er{vexg/h} is a based deformation retract, so is $\ve_\cy(G/H)$. Therefore $\ve_\cx\wedge \varepsilon_\cy$ is a weak equivalence. 
\end{proof}



\begin{proposition}\label{psi&internal-hom} 
For $\cx$ a based $\cg$-CW complex and $\cy$ a based $\cg$-space, one has a weak equivalence of $G$-spaces
\begin{equation}
\label{hom&psi}
\zeta:\Psi\chom(\cx,\cy)\os\sim\to \M_\ct(\Psi\cx,\Psi\cy).
\end{equation}
\end{proposition}

\begin{proof} 
Define $\zeta$ to be the adjoint of the composition
\begin{equation}
\label{def-lambda}
\lambda:\Psi\chom(\cx,\cy)\wedge \Psi\cx\os\varpi\to \Psi(\chom(\cx,\cy)\wedge \cx)\os{\Psi\mu}\to \Psi\cy,
\end{equation}
where $\varpi$ is as in Lemma \ref{descend} and $\mu$ is the counit map in \er{cg-counit}. 


By definition, to show that $\zeta$ is a weak equivalence, we need to show that 
$$
\Phi\zeta:\Phi\Psi \chom(\cx,\cy)\to \Phi\M_\ct(\Psi\cx,\Psi\cy)\os{\er{internalhom1}} =\chom(\Phi\Psi\cx,\Phi\Psi\cy)
$$ 
is a weak equivalence of $\cg$-spaces. 




This follows from the following commutative diagram
$$
\xymatrix{
\Phi\Psi\chom(\cx,\cy)\ar[r]^{\Phi\zeta}\ar[d]^\ve & \chom (\Phi\Psi\cx,\Phi\Psi\cy)\ar[d]^{\chom(\id,\varepsilon_\cy)}\\
\chom(\cx,\cy)\ar[r]^{\chom(\varepsilon_\cx,\id)} & \chom(\Phi\Psi\cx,\cy),
}
$$
where all the other maps are weak equivalences. 
Recall that we assume that $\cx$ is a $\cg$-CW complex, so $\Phi\Psi\cx$ has the homotopy type of a $\cg$-CW complex. Therefore $\ve_\cx:\Phi\Psi\cx\to \cx$  is a homotopy equivalence by the Whitehead theorem \cite[Theorem VI.3.5]{may}, which implies that $\chom(\ve_\cx,\id)$ is a homotopy equivalence. That $\chom(\id,\ve_\cy)$ is a weak equivalence follows from the fact that $\ve_\cy$ is a weak 
equivalence and that $\Phi\Psi\cx$ has the homotopy type of a $\cg$-CW complex by the Whitehead theorem \cite[Theorem VI.3.4]{may}. 
\end{proof}

Now we apply the functor $\Psi$ to a $\cg$-prespectrum to get a $G$-prespectrum. However, for a representation $V$ and the corresponding representation sphere $S^V$, we only have a natural $G$-map 
\begin{equation}
\label{def-eta}
\Theta\ve_V:\Psi\Phi S^V\to S^V,
\end{equation}
which is a weak equivalence of $G$-spaces. Here $\Theta$ is defined in \er{deftheta} (see the proof of Theorem V.3.2 in \cite{may}). To get a $G$-prespectrum in the  sense of \cite{may}, we need to fix the following choices. 

\begin{lemma}
\label{choose-beta}
There is a family of $G$-maps $\beta_V\colon S^V\to\Psi\Phi S^V$, $V\in\ca$, such that
\begin{enumerate}[(i)]
\item $\beta_V$ is a homotopy inverse to $\Theta\varepsilon_V\colon \Psi\Phi S^V\to S^V$;
\item for each $V,W\in\ca$ the following diagram commutes
\begin{equation}
\label{diagram-beta}
\xymatrix{
S^V\wedge S^W\ar[r]^(.4){\beta_V\wedge\beta_W}\ar[d]_{\cong}& \Psi\Phi S^{V}\wedge \Psi\Phi S^W\ar[d]^{\varpi}\\
S^{V+W}\ar[r]_(.4){\beta_{V+W}} & \Psi\Phi S^{V+W},
}
\end{equation}
where $\varpi$ is defined as in Lemma \ref{descend}. 
\end{enumerate}
\end{lemma}

\begin{proof} Recall that $G$ is finite. Choose a representation $V_i$ in each of the finitely many isomorphism classes of irreducible ones. Since $S^{V_i}$ is a $G$-CW complex, $\Psi\Phi S^{V_i}$ has the homotopy type of a $G$-CW complex. By the Whitehead theorem, $\Theta\ve_{V_i}$ in \er{def-eta} is a $G$-homotopy equivalence. Choose and fix an inverse $\beta_i$. A general representation $V\in \ca$ has a fixed decomposition into the irreducible ones $\{V_i\}$. Without loss of generality, assume that $V=V_1+V_2$. Then we define 
$$
\beta_V:S^V=S^{V_1}\wedge S^{V_2}\os{\beta_{V_1}\wedge \beta_{V_2}}\longrightarrow \Psi\Phi S^{V_1}\wedge \Psi\Phi S^{V_2}\os\varpi\to \Psi\Phi S^V.
$$
The general commutativity in diagram \er{diagram-beta} then follows by construction.
\end{proof}

\begin{definition}
Given a $\cg$-prespectrum $\cx=\{\cx_V\}_{V\in \ca}$ with structure map 
\begin{equation}
\label{lastsigma}
 \sigma:\Phi S^W\wedge \cx_V\to \cx_{V+W},
\end{equation}
define $\Psi\cx$ to be the \emph{$G$-prespectrum} whose
value on a representation $V\in \ca$ is 
$$(\Psi\cx)_V:=\Psi\cx_V,$$
and whose structure map for $W\in \ca$ is the composition
\begin{equation}
\label{varsigma}
\varsigma:S^W\wedge (\Psi\cx)_V\os{\beta_W\wedge \id}\to \Psi\Phi S^W\wedge \Psi\cx_V\os{\varpi}\to \Psi(\Phi S^W\wedge \cx_V)\os{\Psi \sigma}\to \Psi\cx_{V+W}=(\Psi\cx)_{V+W}.
\end{equation}
\end{definition}

\begin{lemma}
Together the family  $\{\Psi\cx_V\}$ and the structure maps $\varsigma$ define a $G$-prespectrum.
\end{lemma}

\begin{proof} This is a check of compatibility, which follows from  our choices in Lemma \ref{choose-beta} and the natural associativity of $\varpi$ in Lemma \ref{descend}. We omit the details.  
\end{proof}

\begin{theorem}
\label{lemma:Psi-omega=omega}
If $\cx$ is an $\varOmega$-$\cg$-spectrum 
then $\Psi\cx$ is an $\Omega$-$G$-spectrum.
\end{theorem}

\begin{proof} By Definition \ref{cgstuff}, the adjoint map
$$
\tau:\cx_V\to \chom(\Phi S^W,\cx_{V+W})
$$
associated to the structure map \er{lastsigma} is assumed to be a weak equivalence of $\cg$-spaces. 
Therefore the composition 
\begin{align*}
\xi:\Psi\cx_V\os{\Psi\tau}\to \Psi\chom(\Phi S^W,\cx_{V+W})\os{\zeta}\to & \M_\ct(\Psi\Phi S^W,\Psi\cx_{V+W})\os{\beta_W^*}\to \M_\ct(S^W,\Psi\cx_{V+W})
\end{align*}
is a weak equivalence of $G$-spaces, by Lemma \ref{still-good}, Proposition \ref{psi&internal-hom} and Lemma \ref{choose-beta}. It can be checked that $\xi$ is the adjoint of the structure map $\varsigma$ in \er{varsigma} for the prespectrum $\Psi\cx$, in view of the definition of $\zeta$ as the adjoint of $\lambda$ in \er{def-lambda}. Therefore $\Psi\cx$ is an $\Omega$-$G$-spectrum.
\end{proof}


Combining Theorem \ref{loopspace} and Theorem \ref{lemma:Psi-omega=omega}, we get the following final result, as another application of our stable equivariant abelianization functor.

\begin{theorem}\l{finalthm}
For a based $G$-CW complex $X$, the equivariant prespectrum $\Psi((\varSigma^\infty X)\hot M)$ is an $\Omega$-$G$-spectrum
satisfying $\pi_V^G\Psi((\varSigma^\infty X)\hot M)\cong \widetilde{H}_V^G(X;M)$. 

In particular, $\Psi((\varSigma^\infty S^0)\hot M)$ is an equivariant Eilenberg-Mac~Lane spectrum $H\!{M}$, and for each finite dimensional 
$G$-representation $V$, the $G$-space $\Psi(S^V\hot M)$ is an  equivariant Eilenberg-Mac~Lane space $K(M,V)$.
\end{theorem}

\bibliography{TOPOL-3567}

\providecommand{\bysame}{\leavevmode\hbox to3em{\hrulefill}\thinspace}
\providecommand{\MR}{\relax\ifhmode\unskip\space\fi MR }
\providecommand{\MRhref}[2]{%
  \href{http://www.ams.org/mathscinet-getitem?mr=#1}{#2}
}
\providecommand{\href}[2]{#2}
\begin{thebibliography}{10}

\bibitem{bredon}
Glen~E. Bredon, \emph{Equivariant cohomology theories}, Lecture Notes in
  Mathematics, No. 34, Springer-Verlag, Berlin, 1967.

\bibitem{cw}
Steven~R. Costenoble and Stefan Waner, \emph{Fixed set systems of equivariant
  infinite loop spaces}, Trans. Amer. Math. Soc. \textbf{326} (1991), no.~2,
  485--505.

\bibitem{dt}
Albrecht Dold and Ren{\'e} Thom, \emph{Quasifaserungen und unendliche
  symmetrische {P}rodukte}, Ann. of Math. (2) \textbf{67} (1958), 239--281.
  \MR{MR0097062 (20 \#3542)}

\bibitem{ds}
Pedro~F. dos Santos, \emph{A note on the equivariant dold-thom theorem}, J.
  Pure Appl. Algebra \textbf{183} (2003), no.~1-3, 299--312.

\bibitem{elmendorf}
A.~D. Elmendorf, \emph{Systems of fixed point sets}, Trans. Amer. Math. Soc.
  \textbf{277} (1983), no.~1, 275--284.

\bibitem{gj}
Paul~G. Goerss and John~F. Jardine, \emph{Simplicial homotopy theory}, Progress
  in Mathematics, vol. 174, Birkh\"auser Verlag, 1999.

\bibitem{gm}
J.~P.~C. Greenlees and J.~P. May, \emph{Some remarks on the structure of mackey
  functors}, Proc. Amer. Math. Soc. \textbf{115} (1992), no.~1, 237--243.

\bibitem{lms}
L.~G. Lewis, Jr., J.~P. May, M.~Steinberger, and J.~E. McClure,
  \emph{Equivariant stable homotopy theory}, Lecture Notes in Mathematics, vol.
  1213, Springer-Verlag, Berlin, 1986, With contributions by J. E. McClure.

\bibitem{may}
J.~P. May, \emph{Equivariant homotopy and cohomology theory}, CBMS Regional
  Conference Series in Mathematics, vol.~91, Published for the Conference Board
  of the Mathematical Sciences, Washington, DC, 1996, With contributions by M.
  Cole, G. Comezana, S. Costenoble, A. D. Elmendorf, J. P. C. Greenlees, L. G.
  Lewis, Jr., R. J. Piacenza, G. Triantafillou, and S. Waner.

\bibitem{n}
Zhaohu Nie, \emph{A functor converting equivariant homology to homotopy},
  Bulletin of London Mathematical Society \textbf{39} (2007), no.~3, 499--508.

\bibitem{dieck}
Tammo tom Dieck, \emph{Transformation groups}, de Gruyter Studies in
  Mathematics, vol.~8, Walter de Gruyter \& Co., Berlin, 1987.

\end{thebibliography}
\bibliographystyle{amsplain}

\end{document}